\documentclass[english,12pt,oneside]{amsproc}
\usepackage[english]{babel}
\usepackage{a4wide}
\usepackage{amsthm}
\usepackage{graphics}
\usepackage{amsfonts, amssymb, amscd, amsmath}
\usepackage{latexsym}
\usepackage[matrix,arrow,curve]{xy}
\usepackage{mathabx}%,mathtools}
\usepackage{color}
\usepackage{pbox}
\usepackage{tikz}
\usetikzlibrary{matrix,decorations.pathreplacing,positioning}
\usepackage{hyperref}

\DeclareMathOperator{\Hom}{Hom}

\DeclareMathOperator{\SU}{SU} \DeclareMathOperator{\SO}{SO}
\DeclareMathOperator{\conj}{conj} \DeclareMathOperator{\Sp}{Sp}
\DeclareMathOperator{\Spec}{Spec}
\DeclareMathOperator{\Mat}{Mat} \DeclareMathOperator{\Facets}{Facets}
\DeclareMathOperator{\Subgroups}{Subgroups} \DeclareMathOperator{\odd}{odd}

\newcommand{\simc}{\!\!\sim}
\newcommand{\jump}[1]{\ensuremath \raisebox{1pt}{$#1$}}
\newcommand{\sit}[1]{\ensuremath \raisebox{-1pt}{$#1$}}

\newcommand{\Zo}{\mathbb{Z}}
\newcommand{\Ro}{\mathbb{R}}
\newcommand{\Rg}{\mathbb{R}_{\geqslant 0}}
\newcommand{\Co}{\mathbb{C}}

\newcommand{\Ho}{\mathbb{H}}
\newcommand{\Oo}{\mathbb{O}}

\newcommand{\Zt}{\Zo_2}

\newcommand{\dd}{\partial}

\newcommand{\F}{\mathcal{F}}

\newcommand{\Z}{\mathcal{Z}}
\newcommand{\ZH}{\Z^{\Ho}}

\newcommand{\RP}{\mathbb{R}P}
\newcommand{\CP}{\mathbb{C}P}
\newcommand{\HP}{\mathbb{H}P}
\newcommand{\OP}{\mathbb{O}P}

\newcounter{stmcounter}[section]
\newcounter{thcounter}
\newcounter{problcounter}

\numberwithin{equation}{section}

\theoremstyle{plain}
\newtheorem{cor}[stmcounter]{Corollary}

\newtheorem{thm}[thcounter]{Theorem}

\newtheorem{prop}[stmcounter]{Proposition}
\newtheorem{lem}[stmcounter]{Lemma}
\newtheorem{probl}[problcounter]{Problem}

\theoremstyle{definition}
\newtheorem{defin}[stmcounter]{Definition}

\theoremstyle{remark}
\newtheorem{ex}[stmcounter]{Example}
\newtheorem{rem}[stmcounter]{Remark}
\newtheorem{con}[stmcounter]{Construction}

\begin{document}

\title{Torus action on quaternionic projective plane and related spaces}

\author{Anton Ayzenberg}
\address{Faculty of computer science, Higher School of Economics
%and
%Steklov Mathematical Institute, Moscow, Russia
}
\email{ayzenberga@gmail.com}

\date{\today}
\thanks{This work is supported by the Russian Science Foundation under grant 18-71-00009.}

\subjclass[2010]{Primary 55R91, 57S15, 57S25, 22F30; Secondary 57R91, 57S17, 20G20, 57M60, 20G41}

\keywords{torus action, complexity one, quaternions, octonions, quasitoric manifold, Kuiper--Massey theorem}

\begin{abstract}
For an action of a compact torus $T$ on a smooth compact manifold~$X$ with isolated fixed points the number $\frac{1}{2}\dim X-\dim T$ is called the complexity of the action. In this paper we study certain examples of torus actions of complexity one and describe their orbit spaces. We prove that $\HP^2/T^3\cong S^5$ and $S^6/T^2\cong S^4$, for the homogeneous spaces $\HP^2=\Sp(3)/(\Sp(2)\times\Sp(1))$ and $S^6=G_2/\SU(3)$. Here the maximal tori of the corresponding Lie groups $\Sp(3)$ and $G_2$ act on the homogeneous spaces by the left multiplication. Next we consider the quaternionic analogues of smooth toric surfaces: they give a class of 8-dimensional manifolds
with the action of $T^3$, generalizing $\HP^2$. We prove that their orbit spaces are homeomorphic to $S^5$ as well. We link this result to Kuiper--Massey theorem and some of its generalizations.
\end{abstract}

\maketitle

\section{Introduction}\label{secIntro}

Consider an effective action of the compact torus $T^k$ on a compact smooth manifold $X=X^{2n}$, such that the set of fixed points is finite and non-empty. The number $n-k$ can be shown to be nonnegative; it is called the \emph{complexity of the action}.

Buchstaber and Terzic \cite{BTober,BT,BT2} introduced the theory of $(2n,k)$-manifolds to study the orbit spaces of non-free actions of a compact torus $T^k$ on $2n$-manifolds. Using this theory they proved the homeomorphisms $G_{4,2}/T^3\cong S^5$ and $F_3/T^2\cong S^4$, where $G_{4,2}$ is the Grassmann manifold of complex $2$-planes in $\Co^4$, and $F_3$ is the manifold of complete flags in~$\Co^3$. The torus actions are naturally induced from the standard torus actions on $\Co^4$ and $\Co^3$ respectively. In both cases of $G_{4,2}$ and $F_3$ the complexity of the natural torus action is equal to $1$. Karshon and Tolman proved in \cite{KTmain} that for hamiltonian actions of complexity one, the orbit space is homeomorphic to a sphere provided that the weights of the tangent representation at each fixed point is in general position (see Definition \ref{defGenPosLocStdAct}). This result covers the cases of $G_{4,2}$ and $F_3$.

Now assume that $X^{2n}$ is a quasitoric manifold with the action of $T^n$, and $T^{n-1}\subset T^n$ is a subtorus such that the induced action of $T^{n-1}$ on $X^{2n}$ is in general position. Then the orbit space of the induced action $X^{2n}/T^{n-1}$ is homeomorphic to a sphere~\cite{AyzCompl}.

For an action of the torus on a space $X$ consider the fibration $X\times_TET\stackrel{X}{\to}BT$ and the corresponding Serre spectral sequence
\begin{equation}\label{eqSerreSeq}
E_2^{*,*}\cong H^*(BT)\otimes H^*(X)\Rightarrow H^*(X\times_TET)=H^*_T(X),
\end{equation}
where $ET\stackrel{T}{\to} BT$ is the universal $T$-bundle, $X\times_TET$ is the Borel construction of $X$, and $H^*_T(X)$ is the equivariant cohomology algebra\footnote{We assume that all cohomology rings are taken with $\Zo$ coefficients unless stated otherwise.} of $X$. The space $X$ with a torus action is called \emph{equivariantly formal} in the sense of Goresky--Macpherson \cite{GKM} if its Serre spectral sequence \eqref{eqSerreSeq} degenerates at $E_2$ term. In particular, the spaces with vanishing odd degree cohomology are all equivariantly formal. It is known \cite[Prop.5.8]{Kirw} that manifolds with hamiltonian torus actions are equivariantly formal. Therefore, all the manifolds listed above: Grassmann manifolds, flag manifolds, symplectic manifolds with hamiltonian actions, quasitoric manifolds, are equivariantly formal. This leads to the following question.

\begin{probl}\label{problMain}
Assume the complexity one action of $T=T^{n-1}$ on a compact smooth manifold $X=X^{2n}$ has isolated fixed points and the tangent weights at each fixed point are in general position. Is there a connection between equivariant formality of $X$ and the property that $X/T$ is homeomorphic to a sphere?
\end{probl}

This problem is related to the result of Masuda and Panov \cite{MasPan}, which states that complexity zero action is equivariantly formal if and only if its orbit space is a homology polytope.

We have a non-example supporting the relation between equivariant formality and sphericity of the orbit space. In \cite{AyzMatr} we studied the spaces of isospectral periodic tridiagonal matrices of size $n$. These spaces provide an infinite series of manifolds with torus actions of complexity one satisfying the assumption of the conjecture. For $n\geqslant 4$ these manifolds are not equivariantly formal and their orbit spaces are not spheres.

Besides $G_{4,2}$ and $F_3$ there exist two other natural examples of actions satisfying the assumptions of the problem: these examples appear in the classification of symmetric complexity one torus actions, see \cite{Kur}. These are the $T^2$-action on $S^6=G_2/SU(3)$ and the $T^3$-action on $\HP^2=\Ho^3/\Ho^*\cong \Sp(3)/(\Sp(2)\times\Sp(1))$. Here $G_2$ is considered as the automorphism group of octonion algebra, $S^6$ is the sphere of imaginary octonions of unit length, and $T^2$ is the maximal torus of $G_2$, acting on $G_2$ by left multiplication. The torus $T^3$ acts on $\Ho^3$ by left multiplication of each homogeneous coordinate; it could as well be understood as a maximal torus of $\Sp(3)$ acting on the homogeneous space $\Sp(3)/(\Sp(2)\times\Sp(1))$ by the left multiplication. Notice, however, that the action of $T^3$ on $\HP^2$ has a discrete subgroup $\langle (-1,-1,-1)\rangle\cong \Zt$ as a non-effective kernel. To make the action effective, we consider the action of the torus $T^3/\langle (-1,-1,-1)\rangle\cong T^3$.

Both $\HP^2$ and $S^6$ have vanishing odd degree cohomology hence are equivariantly formal. In this paper we prove

\begin{thm}\label{thmSphere}
Let the maximal compact torus $T^2\subset G_2$ act on $S^6=G_2/\SU(3)$ by the left multiplication. Then $S^6/T^2\cong S^4$.
\end{thm}

This statement is proved by the technique, described in \cite{AyzCompl}. However, to prove the applicability of this technique, we need to describe explicitly a maximal torus of the group $G_2$ of automorphisms of the octonion algebra. This is done in Section \ref{secHalfDim}, where we recall the basic ideas of \cite{AyzCompl} concerning the restrictions of actions of complexity zero to the actions of complexity one. In Section \ref{secHalfDim} we also give an example of a non-equivariantly formal manifold with complexity one torus action, whose orbit space is however homeomorphic to a sphere. This example explains the nontriviality of Problem \ref{problMain}.
%We specify problem \ref{problMain} by incorporating the structure of equivariant skeleton into the question.

Our second result is the following.

\begin{thm}\label{thmHP}
Let a maximal compact torus $T^3\subset \Sp(3)$ act on $\HP^2=\Sp(3)/(\Sp(2)\times\Sp(1))$ by the left multiplication. Then
$\HP^2/T^3\cong S^5$.
\end{thm}

This statement is related to the result of Arnold \cite[Example 4]{Arn} which asserts the homeomorphism 
\begin{equation}\label{eqArnoldThm}
\HP^2/T^1\cong S^7
\end{equation} 
The proof of Theorem \ref{thmHP} relies on the set of ideas, similar to those used by Arnold. In Section~\ref{secLinearAlgebra} we make some preparations related, in particular, to the notion of spectrohedron. Then, in Section \ref{secHP} we prove Theorem \ref{thmHP}, describe the equivariant skeleton of $\HP^2$ and show its connection to equivariant topology of the Grassmann manifold~$G_{4,2}$.

Next we recall the classical Kuiper--Massey theorem \cite{Kui,Mass} which asserts the homeomorphism $\CP^2/\conj\cong S^4$. Here
$\conj$ is the antiholomorphic involution on the complex projective plane, which conjugates all homogeneous coordinates
simultaneously. In \cite{Arn} Arnold discussed this theorem and noticed its closed relation to the homeomorphism \eqref{eqArnoldThm}. These two results were further extended by Atiyah and Berndt \cite{AtBer} who proved their octonionic version, namely, $\OP^2/\Sp(1)\cong S^{13}$, where $\Sp(1)=\SU(2)=S^3$ is the group of unit quaternions.

Note that Kuiper--Massey theorem is not a specific property of the complex projective plane. In Section \ref{secKuiperMassey} we recall
a generalization \cite{Fin} of Kuiper--Massey theorem, which asserts that $X/\conj\cong S^4$ for any smooth compact toric surface $X$. This result can be extended to quasitoric manifolds if one defines the involution $\conj$ in a natural way, see Proposition~\ref{propQtoricInvol}. Due to this observation we started to believe that Theorem \ref{thmHP} and Arnold's homeomorphism \eqref{eqArnoldThm} can also be extended to more general ``quaternionic surfaces''.

Following the work of Jeremy Hopkinson \cite{Hopk} we consider the class of quaternionic analogues of quasitoric manifolds. In \cite{Hopk} these spaces were called \emph{quoric} manifolds; we borrow this terminology. More specifically, we are interested in compact 8-dimensional manifolds, carrying the action of $S^3\times S^3$, which is locally standard in certain sense; and require the orbit space to be diffeomorphic to a polygon. This class of manifolds naturally contains the spaces $\HP^2$ (the $(S^3)^2$-orbit space is a triangle) and $\HP^1\times\HP^1$ (the $(S^3)^2$-orbit space is a square). A lot of care and a lot of preparatory work should be made only to define quoric manifolds and their basic properties, since the acting group is noncommutative: the intuition behind many aspects of quasitoric manifolds may fail in quaternionic case. This big work was done in detail in \cite{Hopk}. Since we don't have an opportunity to give all the definitions and statements, we only provide rough ideas of the constructions of quoric manifolds, and specifically restrict our attention to dimension~$8$.

Our observation, which seems to haven't been covered previously, is the following. We noticed that each $8$-dimensional quoric manifold carries an effective action of $T^3$ (not just the natural action of $T^2\subset S^3\times S^3$, which anyone would expect). Then we have the following generalization of Theorem \ref{thmHP}.

\begin{thm}\label{thmQuoricQuotients}
For any 8-dimensional quoric manifold, its orbit space by the $T^3$-action is homeomorphic to $S^5$.
\end{thm}

However, while the quotient $\HP^2/T^3$ can be understood geometrically by gluing two copies of 5-dimensional spectrohedra along their boundaries (see Proposition \ref{propNtuplesNspace}), the quotients of quoric manifolds for now lack such a description. This situation is similar to the generalization of Kuiper--Massey theorem: while the quotient $\CP^2/\conj$ can be understood as the boundary of 5-dimensional spectrohedron, the quotient spaces $X/\conj$ of general toric surfaces do not have a description in terms of convex geometry. We believe that both toric and quoric surfaces require further study in the same context.

%It is also well-known that a simply connected compact
%smooth 4-manifold with an effective smooth action of T
%2
%is diffeomorphic to S
%4 or
%pCP
%2
%♯qCP2♯r(CP
%1 × CP
%1
%) with p + q + r ≥ 1 ([13]) - Orlik Raymond Actions of the torus on 4-manifolds.

\section{Reductions of half-dimensional actions}\label{secHalfDim}

For a smooth action of $G$ on a smooth manifold $X$ define the
partition of $X$ by orbit types
\[
X=\bigsqcup_{H\in S(G)}X^H.
\]
Here $H$ runs over all closed subgroups of $G$ and
$X^H=\tilde{\lambda}^{-1}(H)=\{x\in X\mid G_x=H\}$.

\begin{defin}
An effective action of $G$ on a compact smooth manifold $X$ is
called \emph{appropriate} if
\begin{itemize}
\item the fixed points set $X^G$ is finite;
\item (adjoining condition) the closure of every connected component of a
fine partition element $X^H$, $H\neq G$, contains a point $x'$ with
$\dim G_{x'}>\dim H$.
\end{itemize}
\end{defin}

In the following we assume that all actions are effective and appropriate.

Assume $T^k$ acts smoothly on a manifold $X=X^{2n}$, and let $x$
be an isolated fixed point. The representation of the torus in the
tangent space $T_xX$ decomposes into the sum of 2-dimensional real representations
\[
T_xX=V(\alpha_1)\oplus\cdots\oplus V(\alpha_n),
\]
where $\alpha_1,\ldots,\alpha_n\in\Hom(T^k,T^1)\cong \Zo^k$, and
$V(\alpha_i)$ it the representation $tv=\alpha(t)\cdot v$ for $v\in\Co\cong\Ro^2$.
In general the weights $\alpha_1,\ldots,\alpha_n$
are defined uniquely up to sign. One needs to impose a stably complex structure on $X$ for the
weights to be defined without sign ambiguity. However, in the following we
do not need stably complex structure; all definitions are invariant under the change of
signs.

\begin{defin}\label{defGenPosLocStdAct}
An appropriate effective action of $T^n$ on $X^{2n}$ is called locally standard
if the weights of the tangent representation at each fixed point form a basis of the lattice
$\Hom(T^n,T^1)\cong\Zo^n$. An appropriate effective action of $T^{n-1}$ on $X^{2n}$ is called
a complexity one action in general position if, for any fixed point, any $n-1$ of $n$ tangent weights $\alpha_1,\ldots,\alpha_n$ are linearly independent.
\end{defin}

The orbit space of any locally standard action is a manifold with corners (the appropriate condition
means that any face of this manifold has a vertex). It is easy to show that the orbit
space of a complexity one action in general position is a topological manifold (see~\cite[Thm.2.10]{AyzCompl}).

\begin{prop}\label{propDiskToSphereGen}
Consider a locally standard action of $T^n$ on $X=X^{2n}$, whose orbit space
is homeomorphic to a disk $D^n$. Assume that $T^{n-1}\subset T^n$ is a subtorus such that
the induced action of $T^{n-1}$ on $X$ is in general position. Then $X/T^{n-1}$ is
homeomorphic to a sphere $S^{n+1}$.
\end{prop}

\begin{proof}
In \cite{AyzCompl} this fact was proved in the case when $X/T^n$ is a simple polytope. The
general statement is completely similar, so we only give an idea of the proof. Consider the induced map
\[
p\colon X/T^{n-1}\to X/T^n\cong D^n.
\]
The map $p$ is the projection to the orbit space, for the residual action of $T^1=T^n/T^{n-1}$ on~$X/T^{n-1}$. Interior points of~$D^n$ correspond to free orbits, while the boundary points
of $D^n$ correspond to point orbits. Hence $X/T^{n-1}$ is the identification space $D^n\times S^1$,
where $S^1$ is collapsed over the boundary $\dd D^n$. This yields the result.
\end{proof}

\begin{con}\label{conT2actsS6}
Consider the standard $T^3$-action on $S^6$, given by
\begin{equation}\label{eqStandS6}
(t_1,t_2,t_3)(r,z_1,z_2,z_3)=(r,t_1z_1,t_2z_2,t_3z_3),
\end{equation}
where $S^6=\{r\in \Ro, z_i\in\Co\mid r^2+|z_1|^2+|z_2|^2+|z_3|^2=1\}$. The
orbit space
\[
S^6/T^3=\{(r,c_1,c_2,c_3)\in\Ro\times\Rg^3\mid r^2+c_1^2+c_2^2+c_3^2=1\}
\]
is a manifold with two corners (sometimes called a ``rugby ball'') homeomorphic to a 3-disk. The subtorus
\[
T^2=\{(t_1,t_2,t_3)\subset T^3\mid t_1t_2t_3=1\}
\]
acts on $S^6$ in general position. Hence, Proposition \ref{propDiskToSphereGen} implies
\begin{equation}\label{eqS6constructionQuot}
S^6/T^2\cong S^4.
\end{equation}
\end{con}

This formula already looks like the proof of Theorem \ref{thmSphere}. However
we need to check that the action of $T^2\subset T^3$ on $S^6$ defined by Construction \ref{conT2actsS6}
is the same as the action in Theorem \ref{thmSphere}.

\begin{prop}\label{propTwoActsOnS6}
The constructed action of $T^2$ on $S^6$ coincides with
the action of a maximal torus $T^2\subset G_2$ on the sphere $G_2/\SU(3)$ of unit imaginary octonions.
\end{prop}

\begin{figure}[h]
\begin{center}
\includegraphics[scale=0.3]{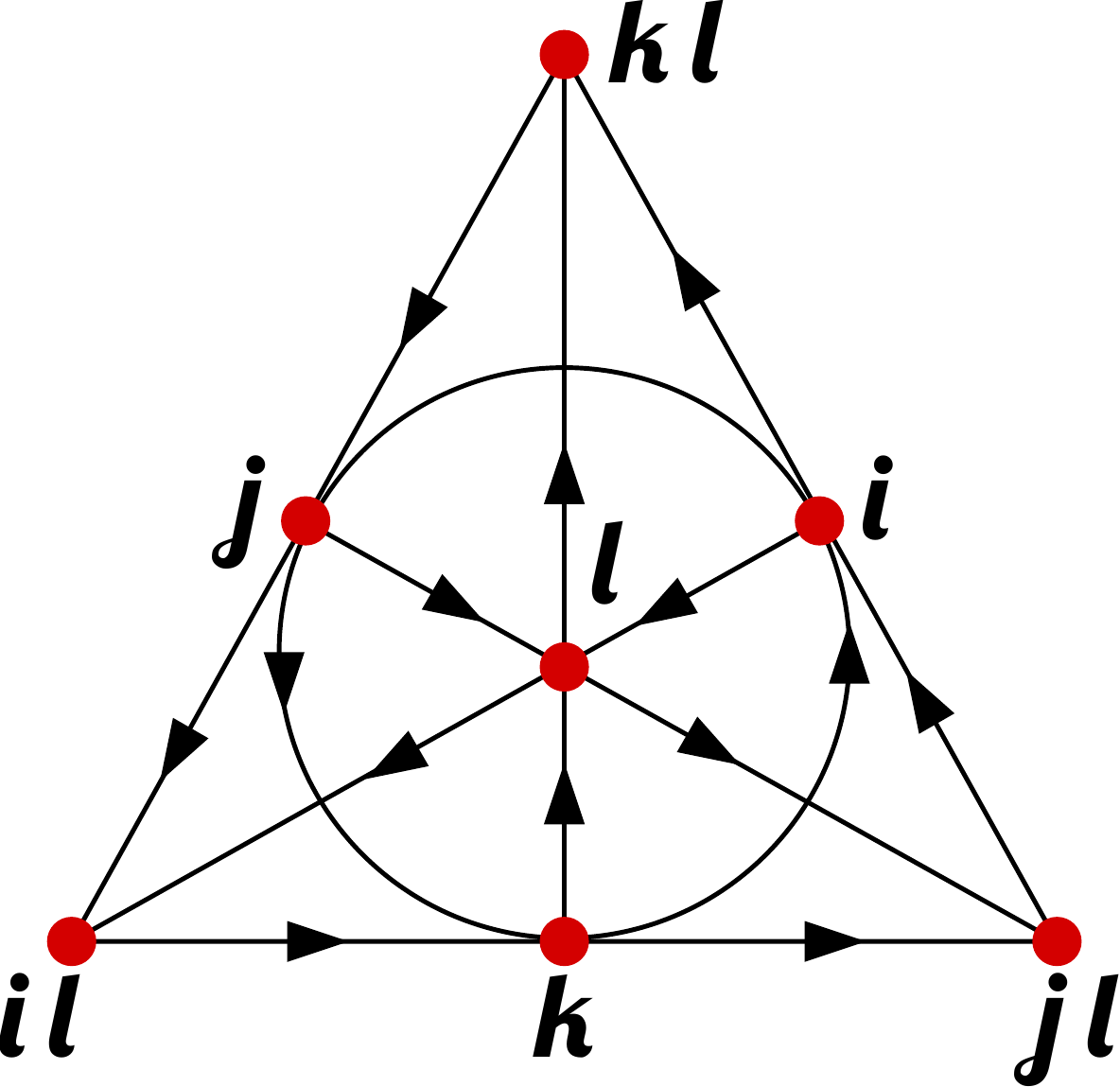}
\end{center}
\caption{Fano plane model for the multiplication of octonionic imaginary units}\label{figOctonionsMult}
\end{figure}

\begin{proof}
Let $1,l,i,j,k,il,jl,kl$ be the standard basis of the octonion algebra $\Oo$ over $\Ro$, with
the multiplication given by the standard mnemonic diagram shown of Fig.\ref{figOctonionsMult}. For each $\alpha,\beta,\gamma\in\Ro/2\pi\Zo$,
$\alpha+\beta+\gamma=0$ consider the automorphism $\sigma_{\alpha,\beta,\gamma}$ of $\Oo$ given by
\[
\sigma_{\alpha,\beta,\gamma}(1)=1;\quad\sigma_{\alpha,\beta,\gamma}(l)=l;\quad \sigma_{\alpha,\beta,\gamma}(i)=e^{\alpha l}i;\quad
\sigma_{\alpha,\beta,\gamma}(j)=e^{\beta l}j;\quad \sigma_{\alpha,\beta,\gamma}(k)=e^{\gamma l}k.
\]
Here we use the notation $e^{\phi \epsilon}=\cos\phi+\epsilon\sin\phi$ for any imaginary
unit $\epsilon=i,j,k$ and $\phi\in\Ro/2\pi\Zo$. A direct calculation shows that $\sigma_{\alpha,\beta,\gamma}$
is indeed an automorphism of $\Oo$. Moreover, we have
\[
\sigma_{\alpha_1+\alpha_2,\beta_1+\beta_2,\gamma_1+\gamma_2}=\sigma_{\alpha_1,\beta_1,\gamma_1}\circ\sigma_{\alpha_2,\beta_2,\gamma_2}.
\]
Hence the torus $T^2_\sigma=\{\sigma_{\alpha,\beta,\gamma}\mid \alpha+\beta+\gamma=0\}$ is
a maximal torus of $G_2$ (it is well known that $G_2$ has rank two, so that the torus cannot have larger dimension). Let us write the automorphism $\sigma_{\alpha,\beta,\gamma}\colon \Ro^8\to\Ro^8$ in the matrix form. In
the basis $1,l,i,il,j,jl,k,kl$ we have
\[
\sigma_{\alpha,\beta,\gamma}=\begin{pmatrix}
  1 &  &  &  &  &  &  &  \\
   & 1 &  &  &  &  &  &  \\
   &  & \cos\alpha & \sin\alpha &  &  &  &  \\
   &  & -\sin\alpha & \cos\alpha &  &  &  &  \\
   &  &  &  & \cos\beta & \sin\beta &  &  \\
   &  &  &  & -\sin\beta & \cos\beta &  &  \\
   &  &  &  &  &  & \cos\gamma & \sin\gamma \\
   &  &  &  &  &  & -\sin\gamma & \cos\gamma
\end{pmatrix}
\]
where void spaces denote zeroes. Since $\alpha+\beta+\gamma=0$, we see that the action of $T^2_\sigma$ on
the $6$-sphere of unit imaginary octonions is exactly the restriction of the standard action \eqref{eqStandS6} of
$T^3$ on $S^6$ to the subtorus $T^2=\{(t_1,t_2,t_3)\in T^3\mid t_1t_2t_3~=~1\}$.
\end{proof}

Proposition \ref{propTwoActsOnS6} and \eqref{eqS6constructionQuot} imply Theorem \ref{thmSphere}.

Now we use the arguments from the beginning of this section to show the nontriviality of Problem \ref{problMain}.

\begin{prop}
There exists a manifold $X=X^{2n}$ with an appropriate complexity one action of $T^{n-1}$ in general position such that the orbit space $X/T^{n-1}$ is homeomorphic to $S^{n+1}$, however the action is not equivariantly formal.
\end{prop}

\begin{proof}
In view of Proposition \ref{propDiskToSphereGen}, the idea of the following construction is very simple: we construct a non-equivariant manifold with half-dimensional torus action, whose orbit space is a disk. For locally standard torus action of $T^n$ on $X^{2n}$ we have a result of Masuda and Panov \cite{MasPan} which asserts the equivalence of the following conditions:
\begin{itemize}
  \item the action of $T^n$ on $X^{2n}$ is equivariantly formal;
  \item the orbit space $X^{2n}/T^n$ is a homology polytope (i.e. all its faces are $\Zo$-acyclic);
  \item $H^{\odd}(X^{2n})=0$.
\end{itemize}

\begin{lem}[{see \cite{SarSten}}]\label{lemBadDisk}
There exists a manifold with corners $P^n$ which satisfies the following properties
\begin{enumerate}
  \item Every face of $P^n$ has a vertex;
  \item $P^n$ is homeomorphic to a disk;
  \item $P^n$ is not a homology polytope;
  \item $P^n$ is the orbit space of some locally standard torus action.
\end{enumerate}
\end{lem}

\begin{proof}
The example of such manifold with corners is shown on Fig.\ref{figTube}: it is homeomorphic
to 3-disk, however one of its facets is non-acyclic. We describe the procedure which allows
to construct many other examples. The details of this procedure can be found in \cite{SarSten}.

\begin{figure}[h]
\begin{center}
\includegraphics[scale=0.3]{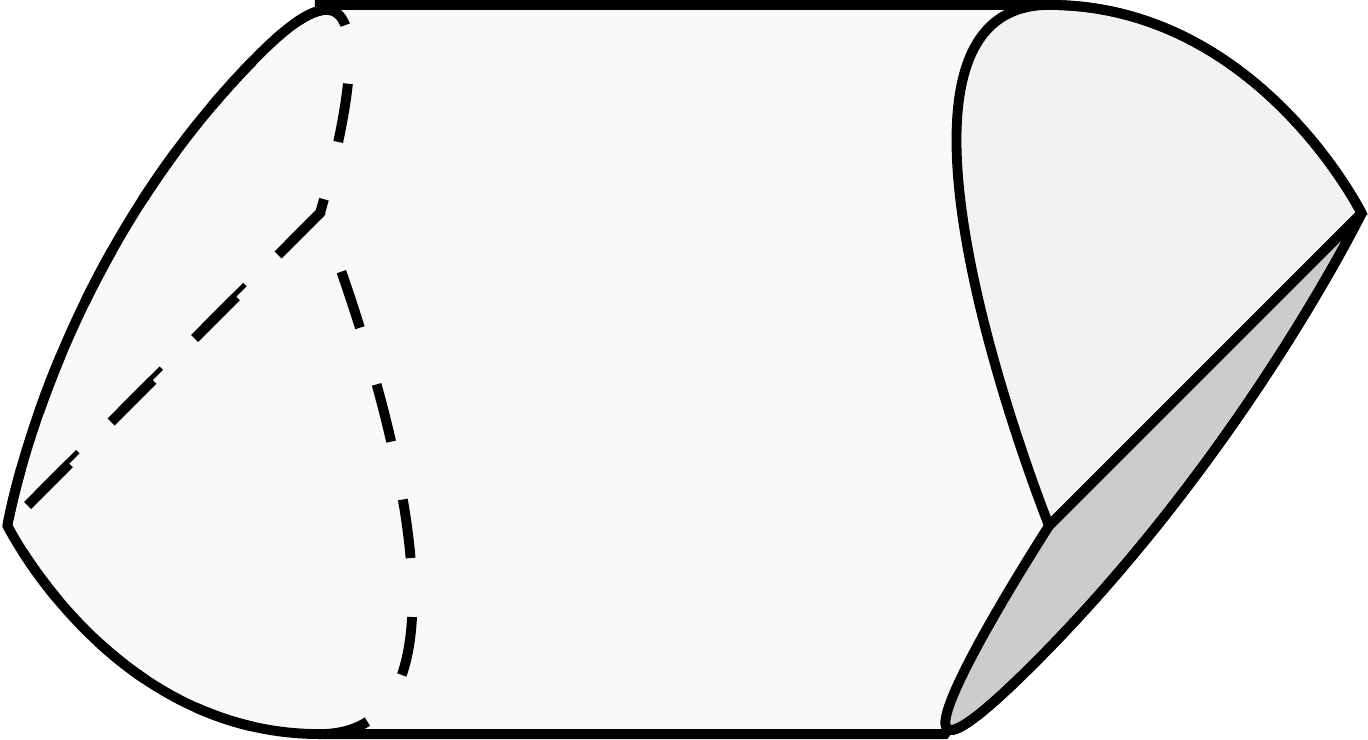}
\end{center}
\caption{A manifold with corners $P^3$, satisfying the properties of Lemma \ref{lemBadDisk}. It is obtained by taking a connected sum of two rugby balls along points in the interior of their facets.}\label{figTube}
\end{figure}

Take any two simple polytopes $P_1^n$ and $P_2^n$, with $n\geqslant 3$. Let $2\leqslant k<n$, and
$x_1$ (resp. $x_2$) be a point lying in the interior of some $k$-dimensional face of $P_1^n$ (resp. $P_2^n$). 
One can form a connected sum $P^n=P_1^n\#_{x_1,x_2}P_2^n$ of manifolds with corners $P_1^n$ and $P_2^n$ along the points $x_1,x_2$. 
The conditions 1--3 are easily checked for $P^n$. Now, if both $P_1^n$ and $P_2^n$ are the orbit spaces of some locally 
standard manifolds $X_1$ and $X_2$, then $P^n$ is the orbit space of the manifold $X_1\#_{T^k}X_2$, where the connected 
sum is taken along some $k$-dimensional torus orbit. For example, the manifold with corners shown on Fig.\ref{figTube} 
is the orbit space of the manifold $S^6\#_{T^2}S^6$ carrying the locally standard action of $T^3$. We refer the reader to \cite{SarSten} where these manifolds were studied in detail.
\end{proof}

If $P^n$ is a manifold with corners, given by Lemma \ref{lemBadDisk}, and $X$ is a torus manifold over $P^n$, then Proposition \ref{propDiskToSphereGen} implies that $X/T^{n-1}\cong S^{n+1}$ for any subtorus $T^{n-1}\subset T^n$ in general position. On the other hand, the result of Masuda and Panov implies that $X$ is not equivariantly formal with respect to the action of $T^n$. However, we shall check that $X$ is not equivariantly formal with respect to the action of $T^{n-1}$ as well.

The latter fact is the consequence of the localization theorem. Indeed, assume the converse, i.e. $X$ is equivariantly formal w.r.t. $T^{n-1}$. This implies $H^*_{T^{n-1}}(X)\cong H^*(BT^{n-1})\otimes H^*(X)$. Lemma \ref{lemBadDisk} and the result of Masuda and Panov imply that $H^*(X)$ has nontrivial odd degree components. Hence $H^*_{T^{n-1}}(X)$ has nontrivial odd degree components as well. Localization theorem asserts the isomorphism of localized modules
\begin{equation}\label{eqLocalizFormula}
S^{-1}H^*_{T^{n-1}}(X)\cong S^{-1}H^*_{T^{n-1}}(X^{T^{n-1}}),
\end{equation}
for some homogeneous multiplicative set $S\subset H^*(BT^{n-1})$. However the fixed point set $X^{T^{n-1}}$ is finite by assumption,
therefore $H^*_{T^{n-1}}(X^{T^{n-1}})\cong H^*(BT^{n-1})^d$. Hence the r.h.s. of \eqref{eqLocalizFormula} does not have odd degree components: this gives a contradiction.
\end{proof}

\section{Vector-tuples up to orthogonal transformations}\label{secLinearAlgebra}

Let $\Mat_{l\times k}$ denote the vector space of real $(l\times k)$-matrices and
\[
Y_{l,k}=\Mat_{l\times k}/\Ro_+\cong S^{lk-1}
\]
denote the sphere of normalized matrices. The orthogonal group $O(l)$ acts on $Y_{l,k}$ by the left multiplication.
In \cite{Arn} Arnold proved

\begin{prop}\label{propNplusOneTuplesNspace}
The orbit space $Y_{n-1,n}/O(n-1)$ is homeomorphic to a sphere of dimension $(n^2+n-4)/2$.
\end{prop}

\begin{proof}
We outline the proof. Each matrix $A\in Y_{n-1,n}$ can be thought as a normalized $n$-tuple of vectors in Euclidean space $\Ro^{n-1}$. With each such tuple we can associate its Gram matrix, i.e. the square matrix $G=A^\top A$ of size $n$. It can be seen that two $n$-tuples of vectors produce the same Gram matrix if and only if the tuples differ by common orthogonal transformation. All Gram matrices produced from $Y_{n-1,n}$ are positive semi-definite symmetric matrices, moreover, they are degenerate. Positive semi-definite symmetric matrices form a strictly convex cone $C_{n}$ in the space of all symmetric matrices of size $n$ (the space and the cone both have dimension $n(n+1)/2$). The boundary of $C_{n}$ consists of degenerate positive semi-definite matrices. Therefore, the space of rays lying in $\dd C_{n}$ is homeomorphic to $Y_{n-1,n}/O(n-1)$. On the other hand, the space of rays, lying in the boundary of any strictly convex cone of dimension $d$ is obviously a sphere of dimension $d-2$.
\end{proof}

It is convenient to intersect the strictly convex cone $C_n$ of positive semidefinite $(n\times n)$-matrices with a generic affine hyperplane. Whenever such an intersection is nonempty and bounded, the intersection is called a \emph{spectrohedron}\footnote{In some sources the term spectrohedron denotes an intersection of the cone $C_n$ with a plane of any dimension, not just hyperplanes.}. A spectrohedron is therefore a compact convex body, defined uniquely up to projective transformations. Denoting the spectrohedron by $\Spec_n$ we have the formula $\dim \Spec_n=n(n+1)/2-1$. Therefore, spectrohedra have dimensions $2,5,9,14$, etc.

\begin{prop}\label{propNtuplesNspace}
The orbit space $Y_{n,n}/\SO(n)=S^{n^2-1}/\SO(n)$ is homeomorphic to a sphere of dimension $d=(n^2+n-2)/2$. The orbit space $Y_{n,n}/O(n)$ is a disk of the same dimension~$d$.
\end{prop}

\begin{proof}
For a matrix $A\in \Mat_{n\times n}$ we consider its polar decomposition $A=QP$, where $Q\in O(n)$ and $P$ is nonnegative symmetric. This decomposition is non-unique if $A$ is degenerate, however, the nonnegative part of the decomposition is uniquely determined by the formula $P=\sqrt{A^\top A}$ for any matrix $A$. The second part of the statement follows easily: the space $Y_{n,n}/\SO(n)$ is homeomorphic to the set of normalized positive semi-definite symmetric matrices of size $n$, which is nothing but the spectrohedron $\Spec_n$, hence a disc of the required dimension.

Let $Z_\geqslant=\{A\in Y_{n,n}\mid \det A\geqslant 0\}$ and $Z_\leqslant=\{A\in Y_{n,n}\mid \det A\leqslant 0\}$. For $A\in Z_\geqslant$ the matrix $Q$ in the polar decomposition can be chosen from $\SO(n)$, therefore $Z_\geqslant/\SO(n)\cong D^d$ as in the previous case. On the other hand, there is an $\SO(n)$-equivariant homeomorphism $Z_\geqslant\to Z_\leqslant$, given by the right multiplication by a fixed reflection matrix. Therefore, we also have $Z_\leqslant/\SO(n)\cong D^d$. The boundaries of both disks are formed by (normalized) degenerate $n$-tuples in $\Ro^n$ considered up to rotations: in degenerate case the matrix $Q$ in the polar decomposition can be taken either with positive or negative determinant, therefore such points belong to both disks. The boundary sphere of both disks is described by Proposition~\ref{propNplusOneTuplesNspace}. Two disks patched along common boundary form a sphere.
\end{proof}

\begin{ex}
Taking $n=2$ in the previous proposition we get $Y_{2,2}/\SO(2)\cong S^2$. The space $Y_{2,2}$ of normalized 2-tuples in $\Ro^2$ coincides with the join of two circles $S^1\ast S^1\cong S^3$, and the circle $S^1=\SO(2)$ acts on the join $S^1\ast S^1$ diagonally. Hence the map $Y_{2,2}\to S^2$ coincides with the Hopf fibration.
\end{ex}

\begin{ex}
For $n=3$ we get $Y_{3,3}/\SO(3)\cong S^5$. The geometrical meaning of this homeomorphism is explained in the next section.
\end{ex}

\section{Quaternionic projective plane}\label{secHP}

Recall that $\HP^2=\Ho^3/\Ho^*=S^{11}/S^3$, where $S^{11}$ is the unit sphere in the space $\Ho^3\cong\Co^6\cong \Ro^{12}$, and $S^3$ is the sphere of unit quaternions. The group $S^3$ is isomorphic to $\SU(2)$ and its action on $\Ho^3$ can be identified with the natural coordinate-wise action of $\SU(2)$ on $(\Co^2)^3$. In the following we assume that the groups $\SU(2)$ act from the right, while the torus $T^3$ acts on $\Ho^3\cong (\Co^2)^3$ by the left multiplication. Each component of the torus acts on the corresponding copy of $\Co^2$ by multiplication.

We now prove Theorem \ref{thmHP}, i.e. the homeomorphism $\HP^2/T^3\cong S^5$.

\begin{proof}[Proof of Theorem \ref{thmHP}]
We need to describe the double quotient
\[
\HP^2/T^3=T^3\backslash \jump{S^{11}}/S^3.
\]
First note that $S^{11}$ can be represented as the join
\[
S^{11}\cong S^3\ast S^3\ast S^3,
\]
where each $S^3$ is the sphere of unit quaternions, taken in the corresponding factor of the product $\Ho\times\Ho\times \Ho$. Therefore,
\[
\sit{T^3}\backslash \jump{S^{11}} \cong (\sit{S^1}\backslash \jump{S^3})\ast(\sit{S^1}\backslash \jump{S^3})\ast(\sit{S^1}\backslash \jump{S^3})\cong \CP^1\ast\CP^1\ast\CP^1
\]
since each factor of the join forms a Hopf bundle $S^3\stackrel{S^1}{\to}\CP^1$. Let us describe the quotient
\[
\sit{T^3}\backslash \jump{S^{11}}/\sit{S^3}\cong (\CP^1\ast\CP^1\ast\CP^1)/\SU(2),
\]
where $\SU(2)$ acts simultaneously on all copies of $\CP^1$ in the standard way. The action of $\SU(2)$ on $\CP^1$ has noneffective kernel $\{\pm1\}$, therefore reduces to the action of $\SU(2)/\{\pm1\}\cong \SO(3)$ on $\CP^1$. This action coincides with the action of the rotation group $\SO(3)$ on the round 2-sphere $S^2\subset \Ro^3$. Hence we need to describe the orbit space
\[
(S^2\ast S^2\ast S^2)/\SO(3)
\]
Note that the join $S^2\ast S^2\ast S^2\cong S^8$ coincides with the sphere $Y_{3,3}$ of normalized $3$-tuples of vectors in $\Ro^3$. Proposition \ref{propNtuplesNspace} implies that the quotient $S^8/\SO(3)$ is homeomorphic to $S^5$. This concludes the proof of Theorem \ref{thmHP}.
\end{proof}

Recall that given an action of a compact torus $T=T^k$ on a manifold $X=X^{2n}$ one can construct an equivariant filtration of $X$:
\[
X_0\subset X_1 \subset X_2\subset\cdots\subset X_k=X,
\]
where the filtration term $X_i$ consists of torus orbits with dimension at most $i$. If $Q=X/T$ denotes the orbit space, we obtain the quotient filtration on $Q$:
\[
Q_0\subset Q_1 \subset Q_2\subset\cdots\subset Q_k=Q,\qquad Q_i=X_i/T
\]

As was mentioned in Section \ref{secHalfDim}, whenever all fixed points of a $T^{n-1}$-action on $X^{2n}$ are isolated, and the weights are in general position, then the orbit space is a topological manifold. Moreover, in this case $\dim Q_i=i$, $\dim X_i = 2i$ for $i\leqslant n-2$, and the $(n-1)$-skeleton $Q_{n-2}$ has a local topological structure, encoded in the notion of a \emph{sponge} \cite{AyzCompl}.

A sponge is a topological space locally modeled by a $(n-2)$-skeleton of the fan $\Delta_{n-1}$ of the toric variety $\CP^{n-1}$. The sponge corresponding to an appropriate action of $T^2$ on a 6-manifold is a 3-valent graph: this is merely the GKM-graph of the action. The sponges corresponding to appropriate actions of $T^3$ on 8-manifolds are locally modeled by the space shown on Fig.\ref{figSponge}. In \cite{AyzCompl} we described the sponges for $F_3$ and $G_{4,2}$ and showed that in some cases the question of extendability of the torus action to the action of a larger torus can be reduced to the question of embeddability of the sponge into a low-dimensional sphere.

\begin{figure}[h]
\begin{center}
\includegraphics[scale=0.2]{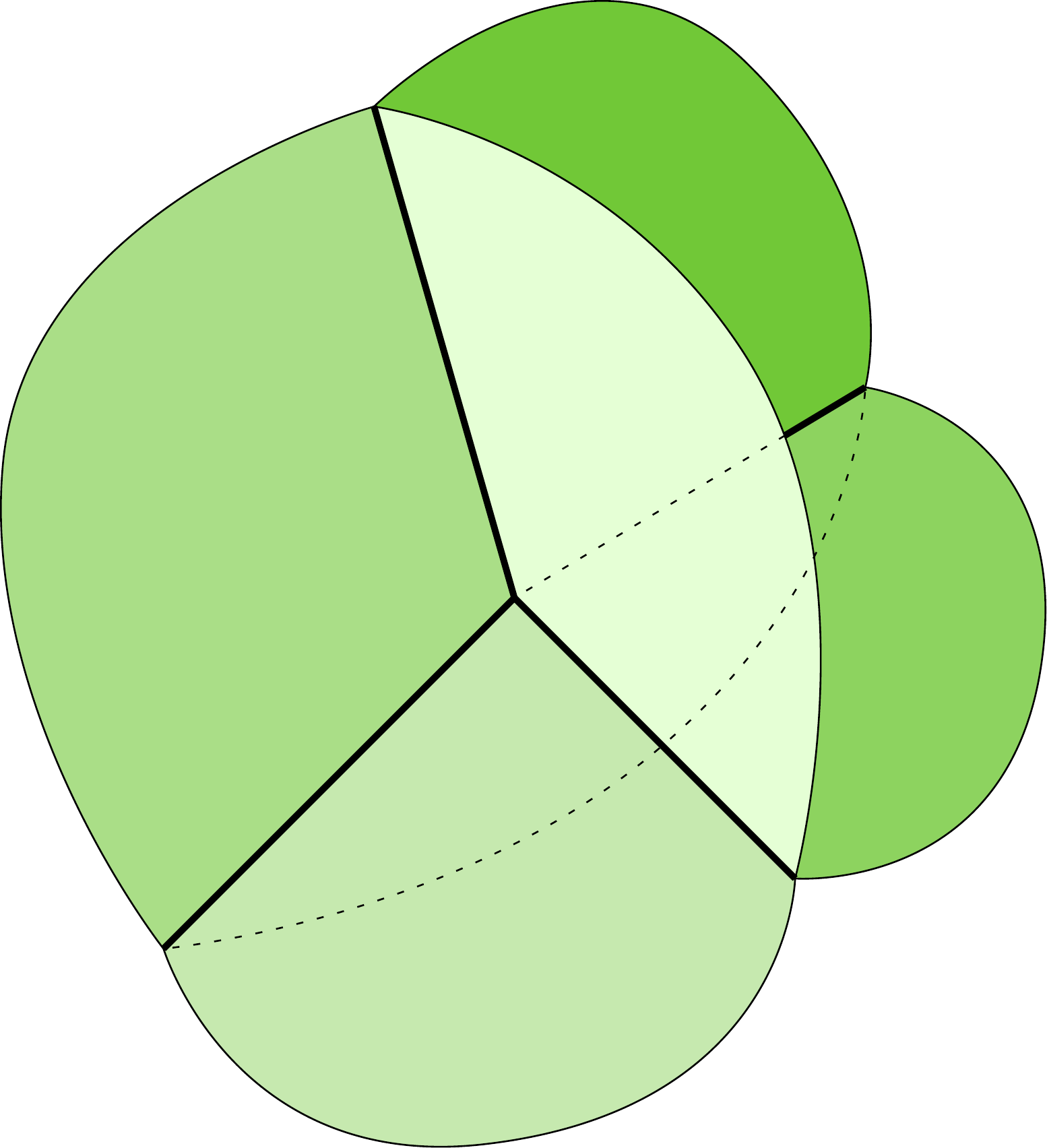}
\end{center}
\caption{Local structure of the sponge for the $T^3$-action on $X^8$.}\label{figSponge}
\end{figure}

Next we describe the torus action on $\HP^2$ and its equivariant skeleton in detail, in order to understand its sponge (the result is shown on Fig.\ref{figHPSponge}). Let $[h_0:h_1:h_2]$ be the homogeneous coordinates on $\HP^2$, defined up to the multiplication by $h\in \Ho^*$ from the right. We represent any quaternion $h\in \Ho$ as $h=z+ju$, and write $h=(z,u)$, where $z,u\in \Co$, and $j$ is the imaginary unit. The letter $t$ denotes the element of 1-dimensional torus: $t\in \Co$, $|t|=1$. As was already mentioned, the torus $T^3$ acts on $\HP^2$ by multiplication from the left:
\[
(t_0,t_1,t_2)[h_0:h_1:h_2]=[t_0h_0:t_1h_1:t_2h_2].
\]
It can be seen that the left action of the circle on $\Ho$ can be written in complex coordinates as follows: $t(z,u)=(tz,t^{-1}u)$.

The subgroup $Z=\langle (-1,-1,-1)\rangle\subset T^3$ acts trivially. So far, to make the action effective, we consider the induced action of $T^3/\langle (-1,-1,-1)\rangle$.

\begin{con}
In $\HP^2$ we have the following torus invariant submanifolds:

\begin{enumerate}
  \item 3 copies of $\HP^1\cong S^4$ given by
\[
M_{01}=\{[\ast:\ast:0]\},\quad M_{02}=\{[\ast:0:\ast]\},\quad M_{12}=\{[0:\ast:\ast]\}
\]
  \item 4 copies of $\CP^2$ given by\footnote{It is convenient to have two different symbols for the same object.}
\[
N_{+++}=\{[(\ast,0):(\ast,0):(\ast,0)]\}\quad(=N_{---}=\{[(0,\ast):(0,\ast):(0,\ast)]\});
\]
\[
N_{++-}=\{[(\ast,0):(\ast,0):(0,\ast)]\}\quad(=N_{--+}=\{[(0,\ast):(0,\ast):(\ast,0)]\});
\]
\[
N_{+-+}=\{[(\ast,0):(0,\ast):(\ast,0)]\}\quad(=N_{-+-}=\{[(0,\ast):(\ast,0):(0,\ast)]\});
\]
\[
N_{+--}=\{[(\ast,0):(0,\ast):(0,\ast)]\}\quad(=N_{-++}=\{[(0,\ast):(\ast,0):(\ast,0)]\}).
\]
  \item 6 copies of $\CP^1$ given by
\[
S_{0+1+}=\{[(\ast,0):(\ast,0):0]\}\quad(=S_{0-1-}=\{[(0,\ast):(0,\ast):0]\});
\]
\[
S_{0+1-}=\{[(\ast,0):(0,\ast):0]\}\quad(=S_{0-1+}=\{[(0,\ast):(\ast,0):0]\});
\]
\[
S_{0+2+}=\{[(\ast,0):0:(\ast,0)]\}\quad(=S_{0-2-}=\{[(0,\ast):0:(0,\ast)]\});
\]
\[
S_{0+2-}=\{[(\ast,0):0:(0,\ast)]\}\quad(=S_{0-2+}=\{[(0,\ast):0:(\ast,0)]\});
\]
\[
S_{1+2+}=\{[0:(\ast,0):(\ast,0)]\}\quad(=S_{1-2-}=\{[0:(0,\ast):(0,\ast)]\});
\]
\[
S_{1+2-}=\{[0:(\ast,0):(0,\ast)]\}\quad(=S_{1-2+}=\{[0:(0,\ast):(\ast,0)]\}).
\]
  \item The fixed points
\[
v_0=[\ast:0:0],\qquad v_1=[0:\ast:0],\quad v_2=[0:0:\ast].
\]
\end{enumerate}
All the listed submanifolds have non-trivial stabilizers in $T^3$. The 4-submanifolds $M_{ij}$ and $N_{\epsilon_0\epsilon_1\epsilon_2}$ consist of at most 2-dimensional torus orbits. A submanifold $M_{ij}$ is stabilized by the circle $\{t_i=t_j=1\}$, and $N_{\epsilon_0\epsilon_1\epsilon_2}$ is stabilized by the circle $\{t_0^{\epsilon_0}=t_1^{\epsilon_1}=t_2^{\epsilon_2}\}$. The 2-spheres $S_{i\epsilon_ij\epsilon_j}$ consist of 1-dimensional orbits. These 2-spheres are the pairwise intersections of the submanifolds $M_{ij}$ and $N_{\epsilon_0\epsilon_1\epsilon_2}$.
\end{con}

\begin{figure}[h]
\begin{center}
\includegraphics[scale=0.3]{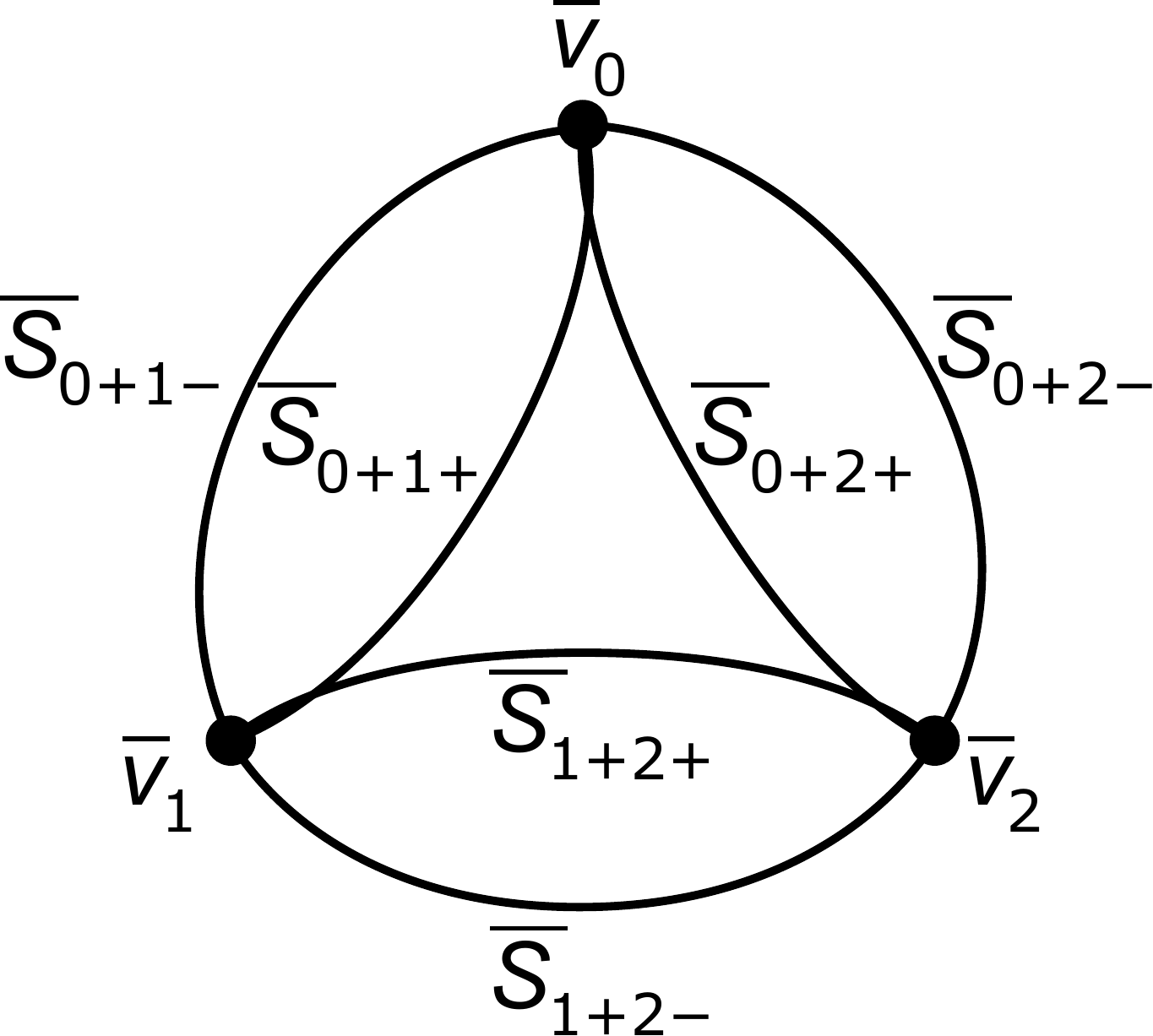}
\end{center}
\caption{The unlabeled GKM-graph $(\HP^2/T^3)_1$ of $\HP^2$}\label{figHPGKM}
\end{figure}

\begin{prop}
The equivariant 2-skeleton $(\HP^2)_2$ of $\HP^2$ is the union of the submanifolds $M_{ij}$ and $N_{\epsilon_0\epsilon_1\epsilon_2}$. The equivariant 1-skeleton $(\HP^2)_1$ is the union of submanifolds $S_{i\epsilon_ij\epsilon_j}$. The action of $T^3/\langle(-1,-1,-1)\rangle$ on $\HP^2$ is free outside $(\HP^2)_2$.
\end{prop}

\begin{proof}
Consider the element $t=(t_0,t_1,t_2)\in T^3$ stabilizing a point $x=[h_0:h_1:h_2]\in \HP^2$. We need to show that whenever $(t_0,t_1,t_2)\notin \langle(-1,-1,-1)\rangle$, the point $[h_0:h_1:h_2]$ lies in one of the subsets $M_{ij}$ or $N_{\epsilon_0\epsilon_1\epsilon_2}$. We may assume $h_0=1$, since otherwise $x\in M_{12}$, and the statement is proved. Since $tx=x$ we have
\[
[1:h_1:h_2]=[t_0:t_1h_1:t_2h_2]=[1:t_1h_1t_0^{-1}:t_2h_2t_0^{-1}].
\]
Therefore $h_1=t_1h_1t_0^{-1}$ and $h_2=t_2h_2t_0^{-1}$. In complex coordinates we have
\[
\begin{cases}
  (z_1,u_1)=(t_1t_0^{-1}z_1,t_1t_0u_1)\\
  (z_2,u_2)=(t_2t_0^{-1}z_2,t_2t_0u_2).
\end{cases}
\]
If at least 3 of the complex numbers $z_1,u_1,z_2,u_2$ are nonzero, the relations on $t_i$ would imply $(t_0,t_1,t_2)=(1,1,1)$ or $(-1,-1,-1)$, contradicting the assumption. Hence we have at least two zeroes among $z_1,u_1,z_2,u_2$, which shows that $x$ lies in either $M_{ij}$ or $N_{\epsilon_0\epsilon_1\epsilon_2}$.

The GKM-graph, i.e. the structure of the set $(\HP^2/T^3)_1$ of at most 1-dimensional orbits, is well known for $\HP^2$, see e.g. \cite{Kur}. The GKM-graph of $\HP^2$ is given by doubling the edges of a triangle (see Fig.\ref{figHPGKM}).
\end{proof}

Recall that the orbit space of the $T^2$-action on $\CP^2$ is a triangle, and the orbit space of the $T^2$-action on $\HP^1\cong S^4$ is a biangle. If $A$ is one of the subsets $M_{ij}$, $N_{\epsilon_0\epsilon_1\epsilon_2}$, $S_{i\epsilon_ij\epsilon_j}$ or $\{v_i\}$, we denote its orbit space by $\overline{A}$. Therefore, $\overline{M_{ij}}$ are biangles, $\overline{N_{\epsilon_0\epsilon_1\epsilon_2}}$ are triangles, and $\overline{S_{i\epsilon_ij\epsilon_j}}$ are closed intervals.

\begin{cor}
The sponge $(\HP^2/T^3)_2$ of the $T^3$-action on $\HP^2$ is obtained by gluing 4 triangles and 3 biangles as shown on Fig.\ref{figHPSponge}.
\end{cor}

\begin{figure}[h]
\begin{center}
\includegraphics[scale=0.3]{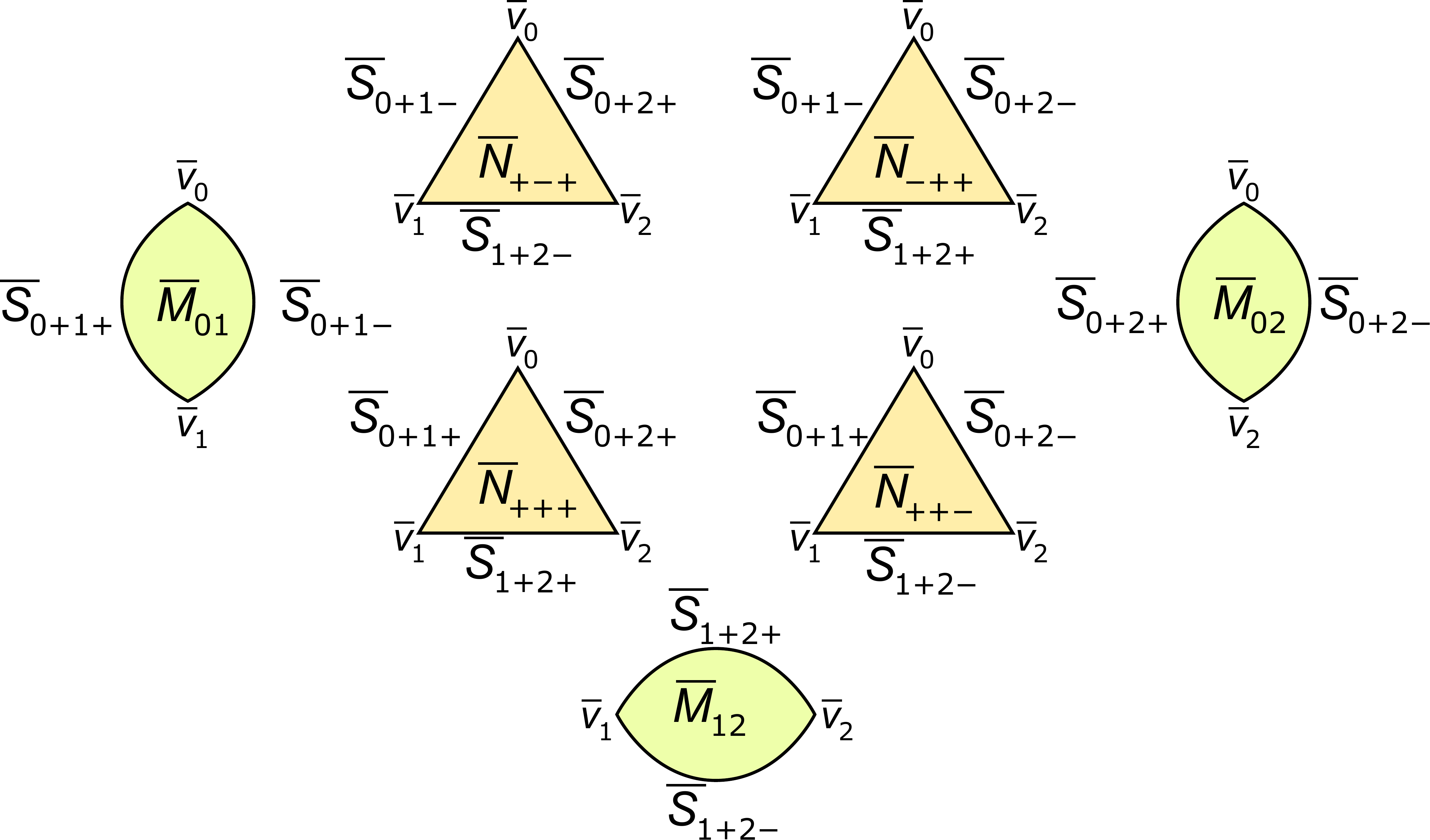}
\end{center}
\caption{The sponge of the $T^3$-action on $\HP^2$ is the cell complex obtained by attaching 4 triangles and 3 biangles}\label{figHPSponge}
\end{figure}

\begin{rem}
There is a simple way to visualize the sponge $(\HP^2/T^3)_2$. First note that 4 triangles in Fig.\ref{figHPSponge} provide the standard triangulation of $\RP^2$, obtained by identifying pairs of opposite points at the boundary of octahedron. So far, to obtain $(\HP^2/T^3)_2$, one needs to attach three disks along three projective lines in $\RP^2$ in general position.
\end{rem}

\begin{figure}[h]
\begin{center}
\includegraphics[scale=0.2]{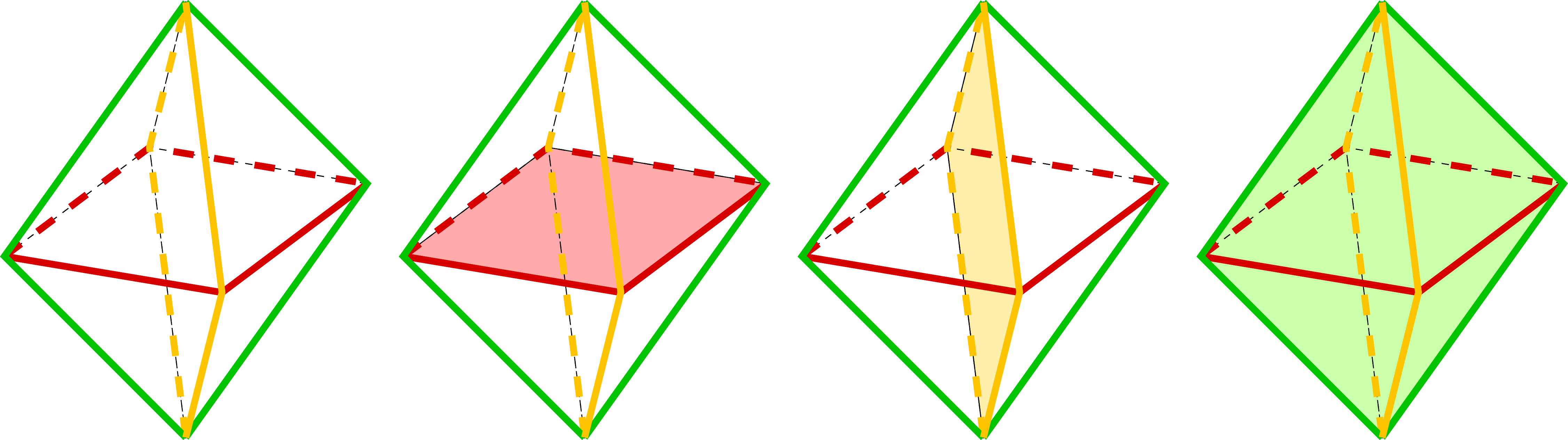}
\end{center}
\caption{The sponge of the $T^3$-action on $G_{4,2}$ is the cell complex obtained by attaching 8 triangles, as in the boundary of an octahedron, and 3 equatorial squares}\label{figGrassmSponge}
\end{figure}

\begin{rem}
There is an interesting connection of $\HP^2$ to the complex Grassmann manifold $G_{4,2}$. Recall \cite{AyzCompl}, that the sponge of the $T^3$-action on $G_{4,2}$ is given by attaching 3 squares to the boundary of an octahedron along equatorial circles, see Fig.\ref{figGrassmSponge}. One can notice that, by identifying pairs of opposite points in this construction, we obtain exactly the sponge $(\HP^2/T^3)_2$ of $\HP^2$.

There is a standard involution $\sigma$ on $G_{4,2}$ which maps a complex 2-plane in $\Co^4$ into its orthogonal complement. This involution induces the antipodal map on the octahedron, the moment map image of $G_{4,2}$. The torus action commutes with $\sigma$, therefore there is a $T^3$-action on the 8-manifold $G_{4,2}/\sigma$. The sponge $((G_{4,2}/\sigma)/T^3)_2$ of this action therefore coincides with the sponge $(\HP^2/T^3)_2$. However, $\HP^2$ is obviously not the same as $G_{4,2}/\sigma$: $\HP^2$ is simply connected, while $G_{4,2}/\sigma$ is not. This example shows that a manifold with complexity one torus action cannot be uniquely reconstructed from its sponge.
\end{rem}

\section{Kuiper--Massey theorem and quasitoric manifolds}\label{secKuiperMassey}

We recall the classical Kuiper--Massey theorem.

\begin{thm}[Kuiper \cite{Kui}, Massey \cite{Mass}]
Let $\conj\colon\CP^2\to\CP^2$ be the antiholomorphic involution of complex conjugation. Then $\CP^2/\conj\cong S^4$.
\end{thm}

On the other hand, similar fact holds not only for $\CP^2$ but for many other 4-manifolds as well. First, we recall several particular examples. 

\begin{ex}\label{exSphereProdInvol}
Let $\sigma\colon S^2\to S^2$ be the reflection in the equatorial plane of a sphere. This involution obviously coincides with $\conj\colon\CP^1\to\CP^1$. Then, considering the involution $\sigma\colon S^2\times S^2\to S^2\times S^2$, which acts on both coordinates simultaneously, it can be proved that
\begin{equation}\label{eqProdSpheresIn}
(\CP^1\times\CP^1)/\conj=(S^2\times S^2)/\sigma\cong S^4.
\end{equation}
See e.g. \cite{Gug}, where a more general collections of involutions on the products of spheres have been considered. One can view
\eqref{eqProdSpheresIn} as the ``complex version'' of the classical homeomorphism
\begin{equation}\label{eq2TorusTo2Sphere}
T^2/\sigma\stackrel{\cong}{\rightarrow} S^2,\mbox{ where } T^2=\Ro^2/\Zo^2\mbox{ and }\sigma(a)=-a,
\end{equation}
given by the Weierstrass's $\wp$-function.
\end{ex}

\begin{ex}\label{exSphereInvol}
Another example is given by the involution $\sigma$ on $S^4$, where
\[
S^4=\{(r,z_1,z_2)\in \Ro\times\Co^2\mid r^2+|z_1|^2+|z_2|^2=1\}
\]
and $\sigma((r,z_1,z_2))=(r,\overline{z}_1,\overline{z}_2)$. It can be seen that 
$S^4=\Sigma(S^1\ast S^1)$, and the involution acts trivially on the first factor of the join (which corresponds to the real parts of $z_i$) and acts freely on the second factor of the join (which corresponds to the imaginary parts of $z_i$). Hence we have, again, $S^4/\sigma\cong \Sigma(S^1\ast (S^1/\sigma))\cong \Sigma(S^1\ast S^1)\cong S^4$.
\end{ex}

Now, according to the result of Orlik and Raymond \cite{OrlRaym}, any simply connected closed 4-manifold $X^4$, acted on by a torus $T^2$, with no nontrivial finite stabilizers, is diffeomorphic to either $S^4$ or an equivariant connected sum of several copies of $\CP^2$, $\CP^1\times\CP^1$, and $\overline{\CP}^2$ (the manifold $\CP^2$ with the reversed orientation) along $T^2$-fixed points. All these spaces carry the natural involutions $\sigma$, and it is not difficult to check that their $T^2$-fixed points have isomorphic tangent $\sigma$-representations (up to orientation reversals). Hence, $X^4$ is represented as a $\sigma$-equivariant connected sum, so its orbit space $X^4/\sigma$ is homeomorphic to $S^4\#\cdots\#S^4\cong S^4$. This, in particular, proves

\begin{thm}[\cite{Fin}]\label{thmFinash}
Let $\conj\colon X^4\to X^4$ be the antiholomorphic involution of complex conjugation on a smooth compact toric surface $X^4$. Then $X^4/\conj\cong S^4$.
\end{thm}

In the focus of toric topology, there lie the notions of a quasitoric manifold and a moment-angle manifold. Next we recall these basic notions and show that there is a natural ``complex conjugation'' on a quasitoric manifold. We prove the analogue of Theorem \ref{thmFinash} without referring to the result of Orlik and Raymond. In Section \ref{secQuoricSurfaces} we extend the analogy to a more interesting 8-dimensional quaternionic case.

\begin{con}
A quasitoric manifold is a manifold $X=X^{2n}$ with a locally standard action of $T^n$, and the orbit space diffeomorphic, as a manifold with corners, to a simple polytope. Each quasitoric manifold determines a characteristic pair $(P,\lambda)$, where $P$ is the polytope of the orbit space, and $\lambda\colon\Facets(P)\to\Subgroups_1(T^n)$ is the map from the set of facets of $P$ to the set of 1-dimensional subgroups of $T^n$. The value $\lambda(\F_i)$ is the stabilizer of any orbit, lying in the interior of a facet $\F_i$. Since the action is locally standard, the characteristic function $\lambda$ satisfies the conditions: (1) each $\lambda(\F_i)$ is a circle; (2) whenever facets $\F_{i_1},\ldots,\F_{i_n}$ of $P$ intersect in a vertex of $P$, the corresponding values $\lambda(\F_{i_1}),\ldots,\lambda(\F_{i_n})$ form a basis of $T^n$ (i.e. the homomorphism $\lambda(\F_{i_1})\times\cdots\times\lambda(\F_{i_n})\to T^n$ induced by inclusions is an isomorphism). The second condition is called $(\ast)$-condition; it obviously implies the first condition.

On the other hand, given a characteristic pair $(P,\lambda)$, it is possible to construct the model space \cite{DJ}
\begin{equation}\label{eqQtoric}
X_{(P,\lambda)}=(P\times T^n)/\simc,
\end{equation}
where the equivalence relation is generated by the relations $(x_1,t_1)\sim (x_2,t_2)$, whenever
$x_1=x_2\in\F_i$ and $t_1^{-1}t_2\in \lambda(\F_i)$. The torus $T^n$ acts on $X_{(P,\lambda)}$ by
rotating the second coordinate. The space $X_{(P,\lambda)}$ is a topological manifold if $\lambda$
satisfies $(\ast)$-condition. According to \cite{DJ}, whenever $X$ is a quasitoric manifold,
$(P,\lambda)$ its characteristic pair, the model space $X_{(P,\lambda)}$ is $T^n$-homeomorphic to the original manifold
$X$. 
\end{con}

Additional considerations are required to construct a smooth structure on $X_{(P,\lambda)}$; we now briefly recall these ideas and refer to \cite{BPR} for details.

\begin{con}
Let $P$ be realized as a convex polytope by affine inequalities
$P=\{x\in\Ro^n\mid Ax+b\geqslant 0\}$, where $A$ is an $(m\times n)$-matrix and $b$ is a column $m$-vector. It
is assumed that no redundant inequalities appear in the definition, so that each of $m$ inequalities
corresponds to a facet of $P$. Consider the affine map $i_A\colon \Ro^n\to \Ro^m$, $i_A(x)=Ax+b$. The
map $i_A$ embeds $P$ into the nonnegative cone $\Rg^m$. The moment-angle space is defined as a pullback in the
diagram
\begin{equation}\label{eqZPdef}
\begin{gathered}
\xymatrix{
\Z_P\ar@{->}[d]\ar@{^{(}->}[r] & \Co^m\ar@{->}[d]^\mu & (z_1,\ldots,z_m)\ar@{|->}[d]^\mu \\
P\ar@{^{(}->}[r]^{i_P} & \Rg^m & (|z_1|,\ldots,|z_m|)}
\end{gathered}
\end{equation}
that is $\Z_P = \mu^{-1}(i_P(P))$. The moment-angle space carries a natural $T^m$-action.
It can be proved \cite{BPR} that $\Z_P$ is a smooth manifold,
whenever $P$ is simple.

Using the characteristic function
$\lambda$ on $P$ one can construct the subgroup $K_\lambda\cong T^{m-n}$ of $T^m$, which acts freely on
$\Z_P$. Then the model space $X_{(P,\lambda)}$ can be defined as
the quotient $\Z_P/K_\lambda$. This construction provides a smooth structure on $X_{(P,\lambda)}$: the one
induced from $\Z_P$.
\end{con}

\begin{con}\label{conInvolQtoric}
Let us construct an analogue of the complex conjugation on a quasitoric manifold. There is 
an involution $\sigma$ on $T^n$:
\[
\sigma\colon (t_1,\ldots,t_n)\mapsto (\overline{t}_1,\ldots, \overline{t}_n).
\]
By extending this involution trivially to $P$, we get an involution $\sigma$ on the product $P\times T^n$.
This involution can be correctly descended to the identification space~\eqref{eqQtoric}. 
Indeed, if $t=(t_1,\ldots,t_n)$ lies in a characteristic subgroup $H\subset T^n$, then, obviously,
$(\overline{t}_1,\ldots, \overline{t}_n)=(t_1^{-1},\ldots,t_n^{-1})$ lies in the same subgroup, so the 
collapsing in \eqref{eqQtoric} respects the involution. We get an involution $\sigma$ on the model space $X_{(P,\lambda)}$.
This is the natural analogue of the involution of conjuguation defined on toric varieties.
\end{con}

\begin{rem}
It is possible to define $\sigma$ as the smooth involution on a quasitoric manifold. One
should first consider the conjugation involution on the moment-angle manifold $\Z_P$, then descend it to $X_{(P,\lambda)}\cong\Z_P/K_\lambda$. We leave details to the reader.
\end{rem}

\begin{prop}\label{propQtoricInvol}
Let $X^4=X_{(P^2,\lambda)}$ be a quasitoric 4-manifold and $\sigma\colon X^4\to X^4$ an involution
defined above. Then $X^4/\sigma\cong S^4$.
\end{prop}

\begin{proof}
We denote the orbit space $X^4/\sigma$ by $Q$. From the construction it follows that there is a map 
$p\colon Q\to P^2$, induced by the projection of $X$ into its $T^2$-orbit space. For the interior point
$x$ of the polygon $P^2$, the preimage $p^{-1}(x)$ is the quotient $T^2/\sigma$, it is homeomorphic to $S^2$ according to \eqref{eq2TorusTo2Sphere}. If $x$ is the vertex of $P^2$, then, obviously, $p^{-1}(x)$ is a point. Finally, if
$x$ lies on a side of $P^2$, then $p^{-1}(x)$ is the quotient of the circle $S^1$ by a non-free involution, hence $p^{-1}(x)$ is
an interval. Therefore, to obtain $Q$ one needs to start\footnote{At this point we essentially use the fact that the quotient map from a quasitoric manifold to its orbit space admits a section.} with $P^2\times S^2$, and pinch the $S^2$-components over the boundary of $P^2$ into either intervals or points. The result is a topological sphere $S^4$.
\end{proof}

\section{Quoric surfaces}\label{secQuoricSurfaces}

Some parts of the story described in Section \ref{secKuiperMassey} remain true when the torus $T^n$ is replaced by its quaternionic analogue, the group $(S^3)^n$, where $S^3=\Sp(1)$ is the sphere of unit quaternions. However, the generalization should be carried with certain care, due to noncommutativity of the group $S^3$. For the detailed exposition of such generalization we refer to the work of Jeremy Hopkinson \cite{Hopk}: he introduced both the concept and the terminology, which we use here. In particular, quaternionic analogues of quasitoric manifolds are called quoric manifolds in \cite{Hopk}. Here we present only general ideas of this theory.

\begin{defin}
A \emph{quoric manifold} is a smooth $4n$-manifold $X^{4n}$ with an action of $(S^3)^n$ satisfying the conditions

(1) The action is locally modelled by one of the ``standard actions'' of $(S^3)^n$ on $\Ho^n$;

(2) The orbit space is diffeomorphic to a simple polytope.
\end{defin}

Here, unlike the toric case, there are several left actions of $(S^3)^n$ on $\Ho^n$ which can
be called ``standard''. For example, we have the following two left actions of $(S^3)^2$ on $\Ho^2$:
\begin{gather}
(s_1,s_2)(h_1,h_2)=(s_1h_1,s_2h_2); \label{eqStdActions1}\\
(s_1,s_2)(h_1,h_2)=(s_1h_1s_2^{-1},s_2h_2). \label{eqStdActions2}
\end{gather}
In both cases the orbit space is the nonnegative cone $\Rg^2$, with the quotient map
given by $(h_1,h_2)\mapsto(|h_1|,|h_2|)$. The orbits corresponding
to the interior of $\Rg^2$ are free, the orbits corresponding to the sides of $\Rg^2$ are
homeomorphic to $S^3$, and the orbit corresponding to the vertex of $\Rg^2$ is a single point.
Therefore, both actions can be considered as the analogues of the standard $T^2$-action on $\Co^2$.

\begin{rem}\label{remDifferentConjclasses}
The actions \eqref{eqStdActions1} and \eqref{eqStdActions2} are not equivalent in any sense, as we now explain.
Notice, that there exist three distinguished subgroups of $S^3\times S^3$, which are isomorphic to $S^3$:
the first coordinate sphere $S^3_{\{1\}}$, the second coordinate sphere $S^3_{\{2\}}$, and the
diagonal subgroup $S^3_{\{1,2\}}=\{(s,s)\mid s\in S^3\}$. All other subgroups isomorphic to $S^3$ lie in
the conjugacy class of $S^3_{\{1,2\}}$. There is no automorphism
of $(S^3)^2$, which takes the conjugacy class of $S^3_{\{1\}}$ to the conjugacy class of $S^3_{\{1,2\}}$ simply
because the conjugacy classes of $S^3_{\{1\}}$ and $S^3_{\{2\}}$ consist of one element each, while the conjugacy
class of $S^3_{\{1,2\}}$ contains infinitely many subgroups.
The actions  \eqref{eqStdActions1} and \eqref{eqStdActions2} are non-equivalent since the subgroup $S^3_{\{1,2\}}$ appears as the
stabilizer in the second action, while it does not appear as the stabilizer in the first action.
\end{rem}

We are interested mainly in quoric 8-manifolds, so it will be enough to mention that \eqref{eqStdActions2}
gives an exhaustive list of the ``standard actions'' of $(S^3)^2$ on $\Ho^2$, up to weak equivalence.

\begin{con}\label{conQuoricModel}
For a quoric manifold $X^{4n}$ one gets a characteristic pair $(P,\Lambda)$, where $P$
is the orbit polytope, and $\Lambda$ is a characteristic functor, defined on
the poset category of faces of $P$ and taking values in conjugacy classes of
subgroups in $(S^3)^n$. This functor satisfies a collection of technical properties \cite{Hopk}.
With these properties satisfied, one can reconstruct a quoric manifold out of characteristic
pair as the model space
\[
X^{4n}_{(P,\Lambda)}=(P\times (S^3)^n)/\simc,
\]
with the identification determined by $\Lambda$ similarly to \eqref{eqQtoric}:
$(x_1,s_1)\sim (x_2,s_2)$ if $x_1=x_2$ and $s_1^{-1}s_2\in \hat{\Lambda}(F(x_1))$. Here
$F(x_1)$ is the unique face of $P$ containing $x_1$ in its interior; and $\hat{\Lambda}(F(x_1))$
denotes the representative in the conjugacy class $\Lambda(F(x_1))$ (it is possible to take all representatives
simultaneously so that the inclusion order is preserved for the representatives). The group $(S^3)^n$
acts on $X^{4n}_{(P,\Lambda)}$ from the left: $s'[(x,s)]=[(x,s's)]$, the action is well defined.

For a general quoric manifold $X^{4n}$ one can reconstruct the model space $X^{4n}_{(P,\Lambda)}$ out
of characteristic pair of $X^{4n}$. The model space is $(S^3)^n$-equivariantly homeomorphic to $X^{4n}$, see~\cite{Hopk}.
\end{con}

For some characteristic functors $\Lambda$ it is possible to obtain the model $X^{4n}_{(P,\Lambda)}$ as
a smooth manifold, using the same principle as in the quasitoric case. Define the quaternionic
moment-angle space as the pullback in the diagram

\begin{equation}\label{eqZPdef}
\begin{gathered}
\xymatrix{
\ZH_P\ar@{->}[d]\ar@{^{(}->}[r]&\Ho^m\ar@{->}[d]^{\mu^\Ho} & (h_1,\ldots,h_m)\ar@{|->}[d]^{\mu^\Ho}\\
P\ar@{^{(}->}[r]^{i_P}&\Ro_{\geqslant}^m & (|h_1|,\ldots,|h_m|)}
\end{gathered}
\end{equation}

Similar to the complex case, $\ZH_P$ is smooth for a simple polytope $P$.
The group $(S^3)^m$ acts on $\Ho^m$ in one of the standard ways so that the map $\mu^{\Ho}$ is identified with 
the projection to the orbit space. Hence $(S^3)^m$ acts on $\ZH_P$,
and its orbit space is the original polytope $P$. If $\Lambda$ satisfies
some additional conditions (axiomatized in the notion of a \emph{global characteristic functor} in \cite{Hopk}), then
it is possible to choose a subgroup $K_\Lambda\subset (S^3)^m$ which is isomorphic to $(S^3)^{m-n}$ and
acts freely on $\ZH_P$. In this case the smooth manifold $X^{4n}_{(P,\Lambda)}$ can be constructed
as the quotient $\ZH_P/K_\Lambda$.

\begin{ex}
The quaternionic projective space $\HP^n$ is the straightforward example of a quoric manifold.
Its corresponding polytope is an $n$-simplex.
\end{ex}

\begin{rem}
Using the Morse-theoretical argument for a polytope $P$ it can be shown that the
cohomology of quoric manifolds are concentrated in degrees divisible by $4$, and there holds
$\dim H^{4k}(X^{4n}_{(P,\Lambda)})=h_k(P)$, where $(h_0(P),h_1(P),\ldots,h_n(P))$ is the $h$-vector
of a simple polytope $P$.
\end{rem}

\begin{con}
We will consider $8$-dimensional quoric manifolds, i.e. quoric manifolds over polygons. According to Construction \ref{conQuoricModel}, one
needs a characteristic pair. A pair consists of an $m$-gon $P^2$, and a characteristic functor, which in this case assigns to any side of $P^2$ one of the three distinguished subgroups of $S^3\times S^3$: either $S^3_{\{1\}}$ or $S^3_{\{2\}}$ (the coordinate spheres) or $S^3_{\{1,2\}}$ (the diagonal sphere). Obviously, in order for the action to be locally standard, different subgroups should be assigned to neighboring sides of $P^2$. This condition is also sufficient for $n=2$. Therefore, a quoric 8-manifold is encoded by polygons with their sides colored in three paints: $\{S^3_{\{1\}}, S^3_{\{2\}}, S^3_{\{1,2\}}\}$, so that the adjacent sides have distinct colors.

Notice that similar combinatorial objects, 3-colored $m$-gons, appear in the classification of 2-dimensional
small covers (see e.g. \cite{Choi} where, in particular, these objects were enumerated). Nevertheless, there is an important difference between small covers and quoric manifolds over polygons even from the combinatorial viewpoint. In 2-dimensional small covers all colors are
interchangeable, since the automorphisms of the real torus $\Zt^2$ form a permutation group $\Sigma_3$.
In quoric case the colors $S^3_{\{1\}}, S^3_{\{2\}}$ are interchangeable, while the color $S^3_{\{1,2\}}$
cannot be interchanged with any of $S^3_{\{1\}}, S^3_{\{2\}}$ by an automorphism of $S^3\times S^3$, as was explained in Remark \ref{remDifferentConjclasses}.
\end{con}

Our interest in quoric 8-manifolds arises from a surprising fact: they provide the
examples of torus actions of complexity one.

\begin{prop}\label{propQuoricToric}
There is an effective action of a compact torus $T^3$ on any quoric 8-manifold $X^8_{(P^2,\Lambda)}$.
This action has $m$ isolated fixed points and the weights of the tangent representation at
each fixed point are in general position.
\end{prop}

\begin{proof}
It is straightforward to get the action of $T^2$: the torus $T^2$ is naturally a
subgroup of $(S^3)^2$. Since $(S^3)^2$ acts on $X^8_{(P^2,\Lambda)}$, the torus $T^2$
acts as well. The additional circle action is, roughly speaking, ``the diagonal action
from the other side''. Let us construct this action globally. At first, we introduce the
action of $T^3=T_1^1\times T_2^1\times T_3^1$ on $(S^3)^2$ by setting
\begin{equation}\label{eqT3actsonS3squared}
(t_1,t_2,t_3)(s_1,s_2)=(t_1s_1t_3,t_2s_2t_3),
\end{equation}
where $s_1,s_2\in S^3$ are unit quaternions,
and $t_i$ are the elements of the coordinate circles $T_i^1$. We get a
$T^3$-action on $P^2\times (S^3)^2$: the torus acts trivially on the first
factor, and acts by~\eqref{eqT3actsonS3squared} on the second factor.

Now we need to prove
that this action descends correctly to the identification space $(P^2\times (S^3)^2)/\simc$.
It is sufficient to prove the correctness only for the action of the third component $T_3^1$, since
the first two components constitute the part of the natural $(S^3)^2$-action. So far we need to
check that, whenever $(x,(s_1,s_2))\sim (x,(\tilde{s}_1,\tilde{s}_2))$ then
$(x,(s_1t_3,s_2t_3))\sim (x,(\tilde{s}_1t_3,\tilde{s}_2t_3))$.

Recalling the definition of the equivalence $\sim$, we have $(s_1^{-1}\tilde{s}_1,s_2^{-1}\tilde{s}_2)\in
\hat{\Lambda}(F(x))$ and we need to prove that $(t_3^{-1}s_1^{-1}\tilde{s}_1t_3,t_3^{-1}s_2^{-1}\tilde{s}_2t_3)\in
\hat{\Lambda}(F(x))$. Hence we need to check that all possible representative subgroups
$\hat{\Lambda}(F(x))\subset (S^3)^2$ are stable under the conjugation by an element $(t_3,t_3)\in T^2\subset (S^3)^2$.
There is a finite number of possibilities to be checked:
\begin{enumerate}
  \item If $x$ lies in the interior of $P^2$, then the subgroup $\hat{\Lambda}(F(x))$ is trivial
  therefore it is stable under a conjugation.
  \item If $x$ is the vertex of $P^2$, then $\hat{\Lambda}(F(x))$ is the whole group $(S^3)^2$, and
  there is nothing to prove as well.
  \item If $x$ lies on the side of $P$, then $\hat{\Lambda}(F(x))$ is one of the groups $S^3_{\{1\}}$,
  $S^3_{\{2\}}$, or $S^3_{\{1,2\}}$. The coordinate subgroups $S^3_{\{1\}}$, $S^3_{\{2\}}$ are stable under any
  conjugation, since their conjugacy classes consist of single elements, see Remark \ref{remDifferentConjclasses}.
  \item Finally, if $\hat{\Lambda}(F(x))=S^3_{\{1,2\}}$, it is a direct check that $S^3_{\{1,2\}}$ is stable:
  \[
  (s,s)\in S^3_{\{1,2\}} \Rightarrow (t_3^{-1}st_3,t_3^{-1}st_3)\in S^3_{\{1,2\}}.
  \]
\end{enumerate}

Hence we get a well-defined $T^3$-action on the model space $X_{(P^2,\Lambda)}$. It can be seen that
its fixed points correspond to the vertices of the polygon.

The statement about the general position of weights follows from the consideration
of the standard actions \eqref{eqStdActions2}. It can be derived that in the first standard action
of $(S^3)^2$ on $\Ho^2$, the corresponding action of $T^3$ on $\Ho^2\cong \Ro^8$ is given by
\[
(t_1,t_2,t_3)(h_1,h_2)=(t_1h_1t_3,t_2h_2t_3);
\]
and for the second standard action, the corresponding action of $T^3$ is given by
\[
(t_1,t_2,t_3)(h_1,h_2)=(t_1h_1t_2^{-1},t_2h_2t_3).
\]
By writing the actions in complex coordinates we get the weights
\[
(1,0,1),(1,0,-1),(0,1,1),(0,1,-1)
\]
in the first case, and the weights
\[
(1,-1,0),(1,1,0),(0,1,1),(0,1,-1)
\]
in the second case. Both vector collections are in general position.
\end{proof}

Now we prove Theorem \ref{thmQuoricQuotients}, which tells that the $T^3$-orbit space of any quoric 8-manifold is homeomorphic to a sphere $S^5$.

\begin{proof}
Given a construction of the quoric manifold as a model $(P^2\times (S^3)^2)/\simc$, it is
natural to prove the theorem by moding out each fiber $(S^3)^2/\simc$ of the map
\[
(P^2\times (S^3)^2)/\simc\,\,\to P
\]
independently. In order to do this we need several lemmas.

\begin{lem}\label{lemS3toInterval}
Consider the action of $T^2$ on the sphere $S^3$ of unit quaternions given by $(t_1,t_2)s=t_1^{\pm1}st_2^{\pm1}$.
Then its orbit space is homeomorphic to a closed interval $D^1$.
\end{lem}

\begin{proof}
Assume for simplicity that the action is $(t_1,t_2)s=t_1st_2$. In complex coordinates, it has the form
\[
T^2\circlearrowleft S^3=\{(z_1,z_2)\in\Co^2\mid |z_1|^2+|z_2|^2=1\},\qquad (t_1,t_2)(z_1,z_2)=(t_1t_2z_1,t_1t_2^{-1}z_2).
\]
Hence the orbit space can be identified with the interval 
\[
D^1=\{(c_1,c_2)\in\Rg^2\mid c_1+c_2~=~1\}
\]
(this is the action of $T^2$ on the moment-angle manifold of the interval).
\end{proof}

\begin{lem}\label{lemS3S3toS3}
Consider the action of $T^3$ on $(S^3)^2$ given by $(t_1,t_2,t_3)(s_1,s_2)=(t_1s_1t_3,t_2s_2t_3)$
or by $(t_1s_1t_2^{-1},t_2s_2t_3)$. Then
its orbit space is homeomorphic to $S^3$.
\end{lem}

\begin{proof}
Consider the larger action of $T^4=T^2\times T^2$ on $S^3\times S^3$, where each $T^2$ acts on the corresponding
$S^3$ as in the previous lemma. Then $(S^3)^2/T^4\cong D^1\times D^1$. The $T^3$-action on $(S^3)^2$ is given
by restricting the $T^4$-action to a certain 3-dimensional subtorus. We have a map
\[
(S^3)^2/T^3\to (S^3)^2/T^4\cong (D^1)^2.
\]
Similar to Proposition \ref{propDiskToSphereGen}, the preimages of interior points in the square $(D^1)^2$
are circles, while the preimages of boundary points are single points. Hence $(S^3)^2/T^3$ is a sphere.
\end{proof}

Now consider the action of $T^3$ on $X^8_{(P^2,\Lambda)}$ introduced in the proof of Proposition \ref{propQuoricToric}. We have a map
\[
p\colon X^8_{(P^2,\Lambda)}/T^3\to X^8_{(P^2,\Lambda)}/(S^3)^2=P^2.
\]
The preimage of the point $x\in P^2$ is the double quotient
\[
p^{-1}(x)=((S^3)^2/\hat{\Lambda}(F(x)))/T^3
\]
If $x$ is a vertex of $P^2$, then $p^{-1}(x)$ is a point. If $x$ lies in the interior of a polygon $P^2$,
then $p^{-1}(x)=(S^3\times S^3)/T^3$ is a sphere $S^3$ according to Lemma \ref{lemS3S3toS3}. If $x$ lies on a side of
a polygon $P^2$, then $(S^3)^2/\hat{\Lambda}(F(x))$ is homeomorphic to a sphere $S^3$ and its quotient by
the residual action of $T^2=T^3/(T^3\cap \hat{\Lambda}(F(x)))$ is an interval according to Lemma \ref{lemS3toInterval}.
Hence we are in a similar situation as the one described in Section \ref{secHalfDim}: the space $X^8_{(P^2,\Lambda)}/T^3$
is obtained from the product $P^2\times S^3$ by pinching the 3-spheres over the boundary $\dd P^2$ into contractible spaces. Hence $X^8_{(P^2,\Lambda)}/T^3\cong S^5$.
\end{proof}

We conclude with the proposition, generalizing Arnold's homeomorphism \eqref{eqArnoldThm}.

\begin{prop}\label{propArnoldForQuoric}
Let $T^1$ be the diagonal subtorus of $T^3$ acting on a quoric $8$-manifold $X^8_{(P^2,\Lambda)}$. Then $X^8_{(P^2,\Lambda)}/T^1\cong S^7$.
\end{prop}

\begin{proof}
We have already constructed the $T^3$-action on $S^3\times S^3$ in the proof of Proposition \ref{propQuoricToric}. The induced action of the diagonal circle $T^1\subset T^3$ on the product $S^3\times S^3=\{(h_1,h_2)\in \Ho^2\mid |h_1|=|h_2|=1\}$ is given by
\begin{equation}\label{eqDiagActionOnS3S3}
t((h_1,h_2))=(th_1t^{-1},th_2t^{-1})
\end{equation}
The proof of the proposition follows the same lines as
the proof of Theorem \ref{thmQuoricQuotients}: we use a sequence of lemmas.

\begin{lem}\label{lemArnLem1}
For the action of $T^1$ on $S^3=\{h\in \Ho\mid |h|=1\}$ given by $t(h)=tht^{-1}$ there holds $S^3/T^1\cong D^2$.
\end{lem}

\begin{lem}\label{lemArnLem2}
For the diagonal action of $T^1$ on $S^3\times S^3$ given by \eqref{eqDiagActionOnS3S3} there holds $(S^3\times S^3)/T^1\cong S^5$.
\end{lem}

Lemma \ref{lemArnLem1} is an exercise similar to Example \ref{exSphereInvol}. The proof of Lemma \ref{lemArnLem2} and the remaining proof of the proposition is completely similar to Lemma \ref{lemS3S3toS3} and Theorem \ref{thmQuoricQuotients} respectively.
\end{proof}

\begin{rem}
It is reasonable to expect that quoric 8-manifolds (at least some of them) can be obtained from $\HP^2$ and $\HP^1\times\HP^1$ by a sequence of equivariant connected sums. In this case, Theorem \ref{thmQuoricQuotients} will be a simple consequence of Theorem \ref{thmHP} and the corresponding statement for $\HP^1\times\HP^1$. In a similar way, Proposition \ref{propArnoldForQuoric} is the consequence of Arnold's result and the corresponding statement for $\HP^1\times\HP^1$.
\end{rem}

\section*{Acknowledgement}
I am grateful to Shintaro Kuroki, who asked the question about the $T^3$-orbit space of $\HP^2$, motivating this work. I thank Victor Buchstaber and Nigel Ray, from whom I knew about Jeremy Hopkinson's work on quaternionic toric topology. I also thank Mikhail Tyomkin for bringing the paper of Atiyah and Berndt on octonionic projective plane to my attention.
%during the conference ``Glances@Manifolds2018'' in Krakow.

\end{document}